\documentclass[12pt,oneside,english]{amsart}
\usepackage[T1]{fontenc}
\usepackage[latin9]{inputenc}
\usepackage{verbatim}
\usepackage{mathrsfs}
\usepackage{amsbsy}
\usepackage{amstext}
\usepackage{amsthm}
\usepackage{amssymb}

\makeatletter
\numberwithin{equation}{section}
\numberwithin{figure}{section}

\usepackage{amscd}
\newtheorem{definition}{Definition}[section]
\newtheorem{theorem}[definition]{Theorem}
\newtheorem{lemma}[definition]{Lemma}
\newtheorem{corollary}[definition]{Corollary}

\usepackage{graphicx}
\usepackage{latexsym}
\usepackage{graphicx}
\usepackage{graphics}
\usepackage{framed}
\usepackage{amsfonts}
\usepackage{epsfig}
\usepackage{verbatim}
\usepackage{setspace}
\usepackage{tikz-cd}
\usepackage[shortcuts]{extdash}
\usepackage[hyphens]{url}
\usepackage{hyperref}

\makeatother

\usepackage{babel}
\begin{document}
\title{An Expansion of the Continuity Property}
\author{Jon M. Corson and Evan M. Lee}
\begin{abstract}
One of the advantages of working with Alexander-Spanier-\v{C}ech type
cohomology theory is the continuity property: For inverse systems
of sufficiently well-behaved spaces, the result of taking the cohomology
of their limit is a direct limit of their cohomologies. However, \v{C}ech
cohomology natively works with presheaves of modules rather than modules
themselves. We define the notion of a system of presheaves for an
inverse system of topological spaces, and show that, under the same
circumstances as the ordinary continuity property, a suitable limit
of the system provides the \v{C}ech cohomology of the inverse limit
of the spaces. We then show one application of this result in comparing
the cohomology of an inverse limit of finite groups to that of the
inverse limit of their classifying spaces.
\end{abstract}

\maketitle
\phantom{a}
\\
Corresponding Author: Evan Lee, emlee4@crimson.ua.edu\\
\\
Authors' affiliation:\\
Department of Mathematics\\
University of Alabama\\
Tuscaloosa, AL, USA
\pagebreak\textbf{}\\
\textbf{Funding Declaration:}\\
There is no funding to declare. \\
\\
\textbf{Competing Interest Declaration:}\\
There are no competing interests to declare. \\
\\
\textbf{Availability of Data and Materials declaration:}\\
There are no data or materials outside this manuscript to declare.\\
 \\
\textbf{Author contributions:}\\
Both authors, Jon M. Corson and Evan M. Lee, produced the manuscript
together.

\pagebreak{}

\section{Preliminaries on Presheaves and \v{C}ech Cohomology}

We begin by expanding some basic definitions and theorems regarding
presheaves from \cite{MR1325242} to a slightly more general case.
When referring to a basis of a topological space, we consider only
those bases which are closed under finite interesction of two or more
elements.

Let $X$ be a topological space, and $\mathcal{B}$ a basis for the
topology of $X$. We may view $\mathcal{B}$ as a category whose objects
are the open subsets $U\in\mathcal{B}$ and whose morphisms are inclusion
maps $i:U\to V$. Then a \emph{presheaf on $\mathcal{B}$ }is a contravariant
functor from this category to some other concrete category~$\mathscr{C}$,
such as sets, vector spaces, or groups, which assigns to each inclusion
map $i$ a restriction map $\Gamma(i):\Gamma(V)\to\Gamma(U)$ such
that for any composition $i\circ j$, $\Gamma(i\circ j)=\Gamma(j)\circ\Gamma(i)$,
and the identity $\text{id}_{U}:U\to U$ satisfies $\Gamma(\text{id}_{U})=\text{id}_{\Gamma(U)}$.
For $\gamma\in\Gamma(U)$, if $V\subset U$ with the inclusion map
$i:V\to U$, then take $\gamma|V$ to mean $\Gamma(i)(\gamma)$, which
is an element of $\Gamma(V)$. In the case where $\mathcal{B}$ consists
of \emph{all} of the open subsets of $X$, we will simply call $\Gamma$
a \emph{presheaf on $X$}.

If $\mathscr{U}$ is a collection of elements of $\mathcal{B}$, a
compatible $\mathscr{U}$ family of $\Gamma$ is an indexed family
$\lbrace\gamma_{U}\in\Gamma(U)\rbrace_{U\in\mathscr{U}}$ such that
for every pair $U,V\in\mathscr{U}$, $\gamma_{U}|U\cap V=\gamma_{V}|U\cap V$.
A presheaf on $X$ is called a\emph{ sheaf} if it also satisfies the
following conditions:
\begin{description}
\item [{(G) The Gluing condition}] Given a collection $\mathscr{U}$ of
open subsets of $X$ with $V=\underset{U\in\mathscr{U}}{\bigcup}U$
and a compatible $\mathscr{U}$ family $\lbrace\gamma_{U}\rbrace_{U\in\mathscr{U}}$,
there is an element ${\gamma\in\Gamma(V)}$ such that $\gamma|U=\gamma_{U}$
for all $U\in\mathscr{U}$.
\item [{(M) The Monopresheaf condition}] Given a collection $\mathscr{U}$
of open subsets of $X$ with $V=\underset{U\in\mathscr{U}}{\bigcup}U$,
for any two elements ${\gamma_{1},\gamma_{2}\in\Gamma(V)}$, if $\gamma_{1}|U=\gamma_{2}|U$
for all ${U\in\mathscr{U}}$, then $\gamma_{1}=\gamma_{2}$. Equivalently
for presheaves of modules or abelian groups, if for some ${\gamma\in\Gamma(V)}$,
$\gamma|U=0$ for all $U\in\mathscr{U}$, then $\gamma=0$. 
\end{description}
Given a basis $\mathcal{B}$, for each presheaf $\Gamma$ of modules
on $\mathcal{B}$ we obtain a new presheaf $\Gamma^{+}$ on $X$ whose
elements are compatible families of $\Gamma$. Given a collection
$\mathscr{U}$ of elements of $\mathcal{B}$, let $\Gamma(\mathscr{U})$
be the collection of compatible $\mathscr{U}$ families of $\Gamma$.
If $\mathscr{V}$ is another collection of elements of $\mathcal{B}$
which refines $\mathscr{U}$ as an open cover, then there is a homomorphism
$\Gamma(\mathscr{U})\to\Gamma(\mathscr{V})$ which assigns to each
compatible $\mathscr{U}$ family $\lbrace\gamma_{U}\rbrace$ the compatible
$\mathscr{V}$ family $\lbrace\gamma_{V}\rbrace$ such that if $V\in\mathscr{V}$
is contained in $U\in\mathscr{U}$, then ${\gamma_{V}=\gamma_{U}|V}$;
this is uniquely defined since $\lbrace\gamma_{U}\rbrace$ is a compatible
family. For any fixed open subset $W$ of $X$, let $\mathscr{U}$
vary over the family of open coverings of $W$ by elements of $\mathcal{B}$;
then the collection $\lbrace\Gamma(\mathscr{U})\rbrace$ is a direct
system of modules, so define $\Gamma^{+}(W):=\varinjlim\lbrace\Gamma(\mathscr{U})\rbrace$.
If $W'\subset W$ and $\mathscr{U}\subset\mathcal{B}$ is an open
covering of $W$, then ${\mathscr{U}':=\lbrace U'\subset U\cap W':U\in\mathscr{U},U'\in\mathcal{B}\rbrace}$
is an open covering of $W'$ by elements of $\mathcal{B}$ which refines
$\mathscr{U}$, giving a homomorphism $\Gamma(\mathscr{U})\to\Gamma(\mathscr{U'})$.
By passing to limits, this gives a homomorphism $\Gamma^{+}(W)\to\Gamma^{+}(W')$.
This defines a presheaf $\Gamma^{+}$ on $X$ which depends only on
the values $\Gamma(U)$ for small open subsets $U$ in $\mathcal{B}$.

Furthermore, given any map $\tau:\Gamma_{1}\to\Gamma_{2}$ (i.e.,
natural transformation) between two presheaves of modules on $\mathcal{B}$,
there is an induced map $\tau^{+}:\Gamma_{1}^{+}\to\Gamma_{2}^{+}$
of presheaves on $X$ defined as follows: Let $\gamma\in\Gamma_{1}^{+}(V)$
be represented by a compatible $\mathscr{U}$ family $\lbrace\gamma_{U}\rbrace_{U\in\mathscr{U}}$
where $\mathscr{U}\subset\mathcal{B}$ covers~$V$; then $\tau^{+}(\gamma)\in\Gamma_{2}^{+}(V)$
is the element represented by the compatible $\mathscr{U}$ family
$\lbrace\tau(\gamma_{U})\rbrace_{U\in\mathscr{U}}$. This determines
a homomorphism $\tau^{+}:\Gamma_{1}^{+}(V)\to\Gamma_{2}^{+}(V)$.
These homomorphisms commute with restriction maps, and thus form a
map of presheaves on $X$. Moreover, it is clear that the $+$~operation
is a functor from the category of presheaves on $\mathcal{B}$ to
the category of presheaves on $X$.

Going in the other direction, if $\Gamma$ is any presheaf on $X$,
then we get a presheaf $\Gamma_{\mathcal{B}}$ on $\mathcal{B}$ by
restricting $\Gamma$ to the subcategory $\mathcal{B}$. This clearly
defines a functor from the category of presheaves on $X$ to the category
of presheaves on $\mathcal{B}$. Observe that if $\Gamma$ is a presheaf
on $X$, then $(\Gamma_{\mathcal{B}})^{+}=\Gamma^{+}$ (since every
collection of open subsets of $X$ has a refinement in $\mathcal{B}$
that covers the same set). Whereas, if $\Gamma$ is a presheaf on
$\mathcal{B}$, then there is a natural map $\alpha:\Gamma\to(\Gamma^{+})_{\mathcal{B}}$
of presheaves on $\mathcal{B}$ which for each $V\in\mathcal{B}$,
assigns to each $\gamma\in\Gamma(V)$ the element of $(\Gamma^{+})_{\mathcal{B}}(V)=\Gamma^{+}(V)$
represented by the compatible family $\lbrace\gamma\rbrace\in\Gamma(\mathscr{V})$,
where $\mathscr{V}=\{V\}$. If $\tau:\Gamma_{1}\to\Gamma_{2}$ is
any map of presheaves on $\mathcal{B}$, then we have a commutative
diagram:

\medskip{}

\noindent \begin{center}
\begin{tikzcd} \Gamma_1 \arrow[d, "\alpha"'] \arrow[r, "\tau"] & \Gamma_2 \arrow[d, "\alpha"] \\  {(\Gamma_1^+)}_{\mathcal B} \arrow[r, "(\tau^{+})_{\mathcal B}"]     & {(\Gamma_2^+)}_{\mathcal B}               \end{tikzcd}
\par\end{center}

\noindent \medskip{}
Using the notion of compatible families, the sheaf axioms can be stated
more briefly as:
\begin{description}
\item [{(G)}] If $V$ is an open subset of $X$ and $\mathscr{U}$ is
an open cover $V$, then the restriction homomorphism $\Gamma(V)\to\Gamma(\mathscr{U})$
is surjective.
\item [{(M)}] If $V$ is an open subset of $X$ and $\mathscr{U}$ is
an open cover $V$, then the restriction homomorphism $\Gamma(V)\to\Gamma(\mathscr{U})$
is injective.
\end{description}
Here are some other equivalent ways of stating condition \textbf{(M)}:

\begin{lemma}\label{(M)conditionslemma}

Let $\Gamma$ be a presheaf of modules on $X$. Then the following
conditions are equivalent: 
\begin{enumerate}
\item[$\boldsymbol{(\text{M})}$] If $\mathscr{U}$ is an open cover of a subset $V\subset X$, then
the restriction map $\Gamma(V)\to\Gamma(\mathscr{U})$ is injective. 
\item[$\boldsymbol{(\text{M})}'$] If $\mathscr{U}$ and $\mathscr{W}$ are open covers of a subset
$V\subset X$ and $\mathscr{W}$ is a refinement of $\mathscr{U}$,
then the restriction map $\Gamma(\mathscr{U})\to\Gamma(\mathscr{W})$
is injective. 
\item[$\boldsymbol{(\text{M})}''$] If $\mathscr{U}$ is an open cover of a subset $V\subset X$, then
the canonical homomorphism $\Gamma(\mathscr{U})\to\Gamma^{+}(V)$
is injective. 
\item[$\boldsymbol{(\text{M})}'''$] The natural map of presheaves $\alpha:\Gamma\to\Gamma^{+}$ is injective. 
\end{enumerate}
\end{lemma}

\begin{proof}

$\boldsymbol{\text{(M)}}\Rightarrow\boldsymbol{\text{(M)}}'$: Let
$\mathscr{U}$ and $\mathscr{W}$ be open covers of an open set $V$,
where $\mathscr{W}$ is a refinement of $\mathscr{U}$, and let $\{\gamma_{U}\}\in\ker[\Gamma(\mathscr{U})\to\Gamma(\mathscr{W})]$.
Then given any $U\in\mathscr{U}$, let $\mathscr{W}'_{U}=\{W\cap U\mid W\in\mathscr{W}\}$
and note that $\{\gamma_{U}\}\in\ker[\Gamma(\mathscr{U})\to\Gamma(\mathscr{W}'_{U})]$,
since $\Gamma(\mathscr{U})\to\Gamma(\mathscr{W}'_{U})$ factors through
$\Gamma(\mathscr{W})$. Thus $\gamma_{U}|_{W'}=0$ for each $W'\in\mathscr{W}'_{U}$
(since $W'\subset U$ for each $W'\in\mathscr{W}'_{U}$), and $\mathscr{W}'_{U}$
is an open cover of $U$. Therefore, if $\Gamma$ satisfies~$\boldsymbol{(\text{M})}$,
then $\gamma_{U}=0$ (each $U\in\mathscr{U}$) and it follows that
$\Gamma(\mathscr{U})\to\Gamma(\mathscr{W})$ is injective.

$\boldsymbol{(\text{M})}'\Rightarrow\boldsymbol{(\text{M})}''$: Condition
$(\text{M})$ is a special case of $(\text{M})'$.

$\boldsymbol{(\text{M})}''\Leftrightarrow\boldsymbol{(\text{M})}'''$:
Let $V\subset X$ be open. Then $\alpha:\Gamma(V)\to\Gamma^{+}(V)$
is the canonical homomorphism corresponding to the open cover $\mathscr{U}=\{V\}$
of $V$. Hence $\alpha$ is injective, if $\Gamma$ satisfies $\boldsymbol{(\text{M})}''$.

$\boldsymbol{(\text{M})}'''\Rightarrow\boldsymbol{(\text{M})}$: Let
$\mathscr{U}$ be an open cover of $V\subset X$. Then $\alpha$ factors
as the composition $\Gamma(V)\to\Gamma(\mathscr{U})\to\Gamma^{+}(V)$.
Thus $\Gamma(V)\to\Gamma(\mathscr{U})$ is injective, if $\Gamma$
satisfies $\boldsymbol{(\text{M})}'''$.

\end{proof}

Consequently, if $\Gamma$ satisfies condition $\boldsymbol{(\text{M})}$
and $V$ is an open subset of $X$, then we can identify the modules
$\Gamma(\mathscr{U})$ ($\mathscr{U}$ an open cover of $V$) with
submodules of $\Gamma^{+}(V)$ so that $\Gamma^{+}(V)=\bigcup_{\mathscr{U}}\Gamma(\mathscr{U})$.

\begin{lemma}\label{completionissheaflemma}~
\begin{description}
\item [{(a)}] If $\Gamma$ is a presheaf of modules on the basis $\mathcal{B}$,
then $\Gamma^{+}$ is a presheaf on $X$ satisfying \textbf{(M)}.
\item [{(b)}] If $\Gamma$ is a presheaf of modules on $X$ satisfying
\textbf{(M)}, then $\Gamma^{+}$ is a sheaf on X.
\end{description}
In particular, for any presheaf $\Gamma$ of modules on $\mathcal{B}$,
the presheaf $\Gamma^{++}$ is a sheaf on $X$.

\end{lemma}

\begin{proof}

It is clear from the definition that $\Gamma^{+}$ is a presheaf on
$X$, so to prove \textbf{(a)} it suffices to show that it satisfies
\textbf{(M)}. Let $V$ be an open subset of $X$ and let $\mathscr{U}$
be an open cover of $V$. Suppose ${\gamma\in\Gamma^{+}(V)}$ is such
that $\gamma|U=0$ for all $U\in\mathscr{U}$. Then $\gamma$ is represented
by some compatible family $\{\gamma_{W}\}\in\Gamma(\mathscr{W})$,
where $\mathscr{W}$ is an open cover of $V$. Moreover, for each
$U\in\mathscr{U}$, there exists a $\mathscr{W}'_{U}\subset\mathscr{B}$
that covers $U$ and is a refinement of $\mathscr{W}$ such that $\{\gamma_{W}\}_{W\in\mathscr{W}}$
is in the kernel of ${\Gamma(\mathscr{W})\to\Gamma(\mathscr{W}'_{U})}$
(since $\gamma|U=0$). Now let $\mathscr{W}'=\bigcup_{U\in\mathscr{U}}\mathscr{W}'_{U}\subset\mathcal{B}$.
Then $\mathscr{W}'$ is a refinement of $\mathscr{W}$ that covers
$V$ such that $\{\gamma_{W}\}_{W\in\mathscr{W}}$ is in the kernel
of $\Gamma(\mathscr{W})\to\Gamma(\mathscr{W}')$. Thus $\gamma=0$
in the direct limit $\Gamma^{+}(V)$.

For \textbf{(b)}, let $\Gamma$ be a presheaf on $X$ satisfying \textbf{(M)}.
To see that $\Gamma^{+}$ is a sheaf, first observe that $\Gamma^{+}$
satisfies \textbf{(M)} as a result of \textbf{(a)}. For \textbf{(G)},
let $\mathscr{U}$ be an open cover of an open set $V$ and $\{\gamma_{U}\}$
a compatible $\mathscr{U}$ family of $\Gamma^{+}$. By Lemma~\ref{(M)conditionslemma},
$\Gamma$ satisfies $\boldsymbol{(\text{M})}''$, so for each $U\in\mathscr{U}$,
there is an open cover $\mathscr{W}_{U}$ of $U$ such that $\gamma_{U}=\{\gamma_{W}\}\in\Gamma(\mathscr{W}_{U})\hookrightarrow\Gamma^{+}(U)$.
Let $U_{1},U_{2}\in\mathscr{U}$ and let $\mathscr{V}=\{W_{1}\cap W_{2}\mid W_{1}\in\mathscr{W}_{U_{1}},W_{2}\in\mathscr{W}_{U_{2}}\}$.
Then $\mathscr{V}$ is an open cover of $U_{1}\cap U_{2}$ which is
a common refinement of $\mathscr{W}_{U_{1}}$ and $\mathscr{W}_{U_{2}}$.
Since $\Gamma(\mathscr{V})\to\Gamma^{+}(U_{1}\cap U_{2})$ is injective
by Lemma~\ref{(M)conditionslemma}, the images of $\gamma_{U_{1}}=\{\gamma_{W}\}_{W\in\mathscr{W}_{U_{1}}}$
and $\gamma_{U_{2}}=\{\gamma_{W}\}_{W\in\mathscr{W}_{U_{2}}}$ in
$\Gamma(\mathscr{V})$ must be equal. That is, $\gamma_{W_{1}}|_{W_{1}\cap W_{2}}=\gamma_{W_{2}}|_{W_{1}\cap W_{2}}$
for all $W_{1}\in\mathscr{W}_{U_{1}}$ and $W_{2}\in\mathscr{W}_{U_{2}}$.
It follows that $\{\gamma_{W}\}_{W\in\mathscr{W}}$ is a compatible
$\mathscr{W}$ family of $\Gamma$, where $\mathscr{W}=\bigcup_{U\in\mathscr{U}}\mathscr{W}_{U}$
is an open cover of $V$. Thus, we have an element $\gamma=\{\gamma_{W}\}\in\Gamma(\mathscr{W})\hookrightarrow\Gamma^{+}(V)$
such that $\gamma|_{U}=\gamma_{U}$ for each $U\in\mathscr{U}$.

\end{proof}

With this established, we will call $\widehat{\Gamma}:=\Gamma^{++}$
the \emph{sheafification} of $\Gamma$. Note that sheafification is
a functor from the category of presheaves on a basis $\mathcal{B}$
to the category of sheaves on $X$. It is easy to see that it is left
adjoint to the restriction functor $\bullet_{\mathcal{B}}$. That
is, if $\Gamma$ is a presheaf on $\mathcal{B}$ and $\Gamma'$ is
a sheaf on $X$, then there is a natural bijection 
\[
\hom(\Gamma,\Gamma'_{\mathcal{B}})\cong\hom(\widehat{\Gamma},\Gamma').
\]

\noindent \medskip{}

The following definition is adapted from \cite{stacks2021stacks}.
Let ${\varphi:X\to Y}$ be a surjective continuous map of topological
spaces, and let $\Gamma_{X}$ be a presheaf on $X$. Then we obtain
a presheaf $\varphi\Gamma_{X}$ on $Y$, called the\emph{ pushforward
of $\Gamma_{X}$}, defined by the formula ${\varphi\Gamma_{X}(U):=\Gamma_{X}(\varphi^{-1}(U))}$.
If ${V\subset U}$, then ${\varphi^{-1}(V)\subset\varphi^{-1}(U)}$,
and the restriction map of the pushforward is the map which makes
the following diagram commute:\medskip{}

\noindent \begin{center}
\begin{tikzcd} \varphi \Gamma_X (U) \arrow[r] \arrow[d, "r"'] & \Gamma_X(\varphi^{-1}(U)) \arrow[d, "r"] \arrow[l, "="'] \\ \varphi \Gamma_X(V) \arrow[r]                  & \Gamma_X(\varphi^{-1}(V)) \arrow[l, "="']                \end{tikzcd}
\par\end{center}

\noindent \medskip{}
Now let $\mathscr{U}$ be any collection of open subsets of $Y$.
Then $\varphi^{-1}(\mathscr{U})=\{\varphi^{-1}(U)\mid U\in\mathscr{U}\}$
is a collection of open subsets of $X$; and since $\varphi$ is onto,
the mapping $U\mapsto\varphi^{-1}(U)$ is a one-to-one correspondence
between $\mathscr{U}$ and $\varphi^{-1}(\mathscr{U})$. Thus an isomorphism
$\varphi\Gamma_{X}(\mathscr{U})\to\Gamma_{X}(\varphi^{-1}(\mathscr{U}))$
is given by $\{\gamma_{U}\}\mapsto\{\gamma_{\varphi^{-1}(U)}\}$,
where $\gamma_{\varphi^{-1}(U)}=\gamma_{U}\in\varphi\Gamma_{X}(U)=\Gamma_{X}(\varphi^{-1}(U))$
for each $U\in\mathscr{U}$. Moreover, these isomorphisms of compatible
families commute with refinement homomorphisms: if $\mathscr{U}$,
$\mathscr{V}$ are collections of open subsets of $Y$ and $\mathscr{V}$
is a refinement of $\mathscr{U}$, then the following diagram commutes:\medskip{}

\noindent \begin{center}
\begin{tikzcd} \varphi \Gamma_X (\mathscr{U}) \arrow[r] \arrow[d, "r"'] & \Gamma_X(\varphi^{-1}(\mathscr{U})) \arrow[d, "r"] \arrow[l, "\cong"'] \\ \varphi \Gamma_X(\mathscr{V}) \arrow[r]                  & \Gamma_X(\varphi^{-1}(\mathscr{V})) \arrow[l, "\cong"']                \end{tikzcd}
\par\end{center}

\noindent \medskip{}
In particular, given an open subset $V$ of $Y$, we have compatible
maps between the direct systems $\{\varphi\Gamma_{X}(\mathscr{U})\}$
running over all open covers of $V$ in $Y$ and $\{\Gamma_{X}(\mathscr{W})\}$
running over all open covers of $\varphi^{-1}(V)$ in $X$. Note that
the latter system generally has more covers, in the sense that $\varphi^{-1}(\mathscr{U})$
is a cover of $\varphi^{-1}(V)$ for each cover $\mathscr{U}$ of
$V$, but not necessarily \emph{every} cover of $\varphi^{-1}(V)$
in $X$ is of that form. This determines a map: 
\[
(\varphi\Gamma_{X})^{+}(V)={\varinjlim\{\varphi\Gamma_{X}(\mathscr{U})\}\to\varinjlim\{\Gamma_{X}(\mathscr{W})\}}=\Gamma_{X}^{+}(\varphi^{-1}(V))=\varphi(\Gamma_{X}^{+})(V)
\]
Furthermore, these homomorphisms form a natural map of pre\-sheaves
$(\varphi\Gamma_{X})^{+}\to\varphi(\Gamma_{X}^{+})$.

\noindent \begin{lemma}\label{injectivepushforwardlemma}

Suppose that the presheaf on $\Gamma_{X}$ satisfies condition \textbf{(M)}.
Then $(\varphi\Gamma_{X})^{+}$ and $\varphi(\Gamma_{X}^{+})$ are
both sheaves on $Y$, and the natural map $(\varphi\Gamma_{X})^{+}\to\varphi(\Gamma_{X}^{+})$
is injective.

\noindent \end{lemma}

\noindent \begin{proof}

It is easy to see that pushforwards preserve the conditions \textbf{(M)}
and \textbf{(G)}. Hence the first part follows from Lemma~\ref{completionissheaflemma}~(b).
For the second part, let $\mathscr{U}$ be an open cover of a subset
$V$ of $Y$. Then the composition 
\[
\varphi\Gamma_{X}(\mathscr{U})\stackrel{\cong}{\to}\Gamma_{X}(\varphi^{-1}(\mathscr{U}))\to\Gamma_{X}^{+}(\varphi^{-1}(V))=\varphi(\Gamma_{X}^{+})(V)
\]
is an injective homomorphism by condition $(\text{M})''$ of Lemma~\ref{(M)conditionslemma}.
Hence the homomorphism $(\varphi\Gamma_{X})^{+}(V)=\varinjlim_{\mathscr{U}}\varphi\Gamma_{X}(\mathscr{U})\to\varphi(\Gamma_{X}^{+})(V)$
is also injective. 

\noindent \end{proof}

\subsection*{\v{C}ech Cohomology:\protect \\
}

The standard definition of \v{C}ech Cohomology, as seen in \cite{MR1325242}
for instance, concerns itself only with presheaves on $X$; here we
extend this definition to presheaves on other bases, and show that
the usual theorems still hold in this context. 

Let $\mathcal{B}$ be a basis of $X$, $\Gamma$ a presheaf of abelian
groups on $\mathcal{B}$, and $\mathscr{U}$ an open cover of $X$
consisting of elements of $\mathcal{B}$. For $n\geq0$ define $\check{{C}}^{n}(\mathscr{U},\Gamma)$
to be the module of functions $f$ which assign to any ordered ${(n+1)}$\=/tuple
$U_{0},U_{1},...,U_{n}$ of elements of $\mathscr{U}$ an element
${f(U_{0},U_{1},...,U_{n})}$ in \\
${\Gamma(U_{0}\cap U_{1}\cap...\cap U_{n})}$. The coboundary ${\partial:\check{{C}}^{n}(\mathscr{U},\Gamma)\to\check{{C}}^{n+1}(\mathscr{U},\Gamma)}$
is defined by 
\[
{\partial f(U_{0},...,U_{n+1})}={\underset{i=0}{\overset{n+1}{\sum}}(-1)^{i}f(U_{0},...,\hat{U}_{i},...,U_{n+1})|_{(U_{0}\cap...\cap U_{n+1})}}
\]
 where $\hat{U}_{i}$ denotes omission of $U_{i}$. Since $\partial\partial=0$,
this makes $\check{{C}}^{*}(\mathscr{U},\Gamma)$ a cochain complex,
with cohomology groups $\check{{H}}^{*}(\mathscr{U},\Gamma)$.

Let $\mathscr{V}\subset\mathcal{B}$ be a refinement of $\mathscr{U}$
and $\lambda:\mathscr{V}\to\mathscr{U}$ a function such that $V\subset\lambda(V)$
for all $V\in\mathscr{V}$. This gives a cochain map ${\lambda^{*}:\check{{C}}^{*}(\mathscr{U},\Gamma)\to\check{{C}}^{*}(\mathscr{V},\Gamma)}$
defined as ${\lambda^{*}f(V_{0},...,V_{n})}{=f(\lambda(V_{0}),...,\lambda(V_{n}))|_{(V_{0}\cap...\cap V_{n})}}$.
If $\mu:\mathscr{V}\to\mathscr{U}$ is another such function, a cochain
homotopy \\
${D:\check{{C}}^{n}(\mathscr{U},\Gamma)\to\check{{C}}^{n-1}(\mathscr{V},\Gamma)}$
between $\lambda^{*}$ and $\mu^{*}$ is defined by 
\[
Df(V_{0},...,V_{n-1})=\underset{j=0}{\overset{n-1}{\sum}}(-1)^{j}f(\lambda(V_{0}),...,\lambda(V_{j}),\mu(V_{j}),...,\mu(V_{n-1}))|_{(V_{0}\cap...\cap V_{n-1})}
\]
Thus, there is a well-defined homomorphism $\lambda^{*}:\check{{H}}^{*}(\mathscr{U},\Gamma)\to\check{{H}}^{*}(\mathscr{V},\Gamma)$
with $\lambda^{*}[f]=[\lambda^{*}f]$ which is independent of the
choice of such $\lambda$. These maps $\lambda^{*}$ determined by
refinement make $\lbrace\check{{H}}^{*}(\mathscr{U},\Gamma):{\mathscr{U}\subset\mathcal{B}}$
$\text{ is an open cover of }X\rbrace$ a direct system, and the \emph{\v{C}ech
cohomology} of $X$ with coefficients in $\Gamma$ is defined to be
$\check{{H}}^{*}(X,\Gamma):=\underset{\mathscr{U}\subset\mathcal{B}}{\varinjlim}\check{{H}}^{*}(\mathscr{U},\Gamma)$.
Furthermore, if $A$ is an abelian group, then the \emph{\v{C}ech
cohomology of $X$ with coefficients in $A$} is defined as follows:
Take $\mathcal{B}$ to consist of \emph{all} open subsets of $X$
and $A_{X}$ to be the constant presheaf; then ${\check{{H}}^{*}(X,A):=\check{{H}}^{*}(X,A_{X})}$.

Given a basis $\mathcal{B}$ of $X$, an open cover in $\mathcal{B}$
\emph{indexed by} $X$ is a function $U:X\to\mathcal{B}$ where $x\in U(x)=U_{x}$
for each $x\in X$; the collection $\mathscr{U}=\{U_{x}:x\in X\}$
is clearly an open cover of $X$. For covers indexed by $X$ we have
a preorder $\leq$ defined by $\mathscr{U}\leq\mathscr{V}$ if and
only if $V_{x}\subset U_{x}$ for every $x\in X$. This makes $\mathscr{V}$
a refinement of $\mathscr{U}$, but is a stronger condition, and in
fact makes covers indexed by $X$ a directed partially ordered set.
Furthermore, if $\mathscr{V}$ is any open cover of $X$ and $\mathscr{U}$
is a cover indexed by $X$, we may choose for each $x$ some $V_{x}\in\mathscr{V}$
with $x\in V_{x}$ and then define a common refinement $\mathscr{W}=\{U_{x}\cap V_{x}:x\in X\}$
of both $\mathscr{U}$ and $\mathscr{V}$ which is also indexed by
$X$. This makes any direct limit determined by all open covers equal
to the direct limit determined by covers indexed by $X$, so that
in particular $\check{{H}}^{*}(X,\Gamma)\cong\underset{\mathscr{U}\text{ indexed by }X}{\varinjlim}\check{{H}}^{*}(\mathscr{U},\Gamma)$.

If $\mathscr{U}\leq\mathscr{V}$ are indexed by $X$, this gives a
map $\gamma:\mathscr{V}\to\mathscr{U}$ defined by $\gamma(V_{x})=U_{x}$
for each $x\in X$, which then induces a canonical cochain map $\gamma^{*}:\check{{C}}^{*}(\mathscr{U},\Gamma)\to\check{{C}}(\mathscr{V},\Gamma)$.
We may use this to define a cochain complex $\check{{C}}^{*}(X,\Gamma):=\underset{\mathscr{U}\text{ indexed by }X}{\varinjlim}\check{{C}}^{*}(\mathscr{U},\Gamma)$.
Since the direct limit functor is exact, it commutes with cohomology,
so $\underset{\mathscr{U}\text{ indexed by }X}{\varinjlim}\check{{H}}^{*}(\mathscr{U},\Gamma)$
is isomorphic to the cohomology of this cochain complex; that is,
$\check{{H}}^{*}(X,\Gamma)$ as defined above is isomorphic to the
cohomology groups of $\check{{C}}^{*}(X,\Gamma)$.

\begin{lemma}\label{exactsequencecechcohomologylemma}

For a basis $\mathcal{B}$ of $X$, there is a covariant functor from
the category of short exact sequences of presheaves on $\mathcal{B}$
to the category of exact sequences which assigns to any short exact
sequence $0\to\Gamma'\to\Gamma\to\Gamma''\to0$ of presheaves on $\mathcal{B}$
an exact sequence $...\to\check{{H}}^{n}(X,\Gamma')\to\check{{H}}^{n}(X,\Gamma)\to\check{{H}}^{n}(X,\Gamma'')\to\check{{H}}^{n+1}(X,\Gamma')\to...$

\end{lemma}

\begin{proof}

For any open cover $\mathscr{U}\subset\mathcal{B}$ there is a short
exact sequence of cochain complexes $0\to\check{{C}}^{*}(\mathscr{U},\Gamma')\to\check{{C}}^{*}(\mathscr{U},\Gamma)\to\check{{C}}^{*}(\mathscr{U},\Gamma'')\to0$.
This gives an exact cohomology sequence for $\mathscr{U}$, and the
result follows from passing this to the direct limits defining the
\v{C}ech cohomology groups.\end{proof}

A presheaf of modules $\Gamma$ on $\mathcal{B}$ is \emph{locally
zero} if, for any $\gamma\in\Gamma(V)$ there exists an open cover
$\mathscr{U}\subset\mathcal{B}$ of $V$ with $\gamma|U=0$ for all
$U\in\mathscr{U}$; this is equivalent to $\Gamma^{+}$ being the
zero presheaf (which satisfies $\Gamma(U)=0$ for all $U$) and to
the condition that for every $x\in X$, the stalk $\Gamma_{x}=0$.
A homomorphism $\tau:\Gamma_{1}\to\Gamma_{2}$ between presheaves
on $X$ is called \emph{locally injective} if its kernel is locally
$0$, and a \emph{local isomorphism} if both its kernel and its cokernel
are locally $0$. An open cover on a topological space $X$ is \emph{locally
finite} if for every $x\in X$ there exists an open subset $U$ of
$X$ with $x\in U$ which has nonempty intersection with only finitely
many elements of the cover; a space $X$ is called \emph{paracompact}
if every open cover admits a refinement which is locally finite. In
particular, compact spaces are also paracompact. For an open cover
$\mathscr{U}$, let $\mathscr{U}^{*}:=\{U^{*}\}_{U\in\mathscr{U}}$
where $U^{*}=\cup\{U'\in\mathscr{U}:U'\cap U\neq\emptyset\}$. Another
open cover $\mathscr{V}$ is a \emph{star refinement} of $\mathscr{U}$
if $\mathscr{V}^{*}$ is a refinement of $\mathscr{U}$.

\begin{lemma}\label{locallyzeroiszerocechlemma}

If $X$ is a paracompact Hausdorff space with basis $\mathcal{B}$
and $\Gamma$ is a locally zero presheaf on $\mathcal{B}$, then $\check{{H}}^{*}(X,\Gamma)=0$.

\end{lemma}

\begin{proof}

Let $\mathscr{U}\subset\mathcal{B}$ be a locally finite open covering
of $X$ and $f$ an $n$-cochain of $\check{{C}}^{*}(\mathscr{U},\Gamma)$.
Let $\mathscr{W}$ be a locally finite open star refinement of $\mathscr{U}$.
For each $x\in X$, since $\Gamma$ is locally zero, there is an open
neighborhood $V_{x}\in\mathcal{B}$ contained in some element $W_{x}\in\mathscr{W}$
such that if $x\in U_{0}\cap...\cap U_{n}$ with $U_{0},...,U_{n}\in\mathscr{U}$,
then $V_{x}\subset U_{0}\cap...\cap U_{n}$ and $f(U_{0},...,U_{n})|V_{x}=0$;
this is only a finite number of conditions since $\mathscr{U}$ is
locally finite. Let $\mathscr{V}:=\{V_{x}\}_{x\in X}$ and define
$\lambda:\mathscr{V}\to\mathscr{U}$ so that for each $x\in X$, we
have $V_{x}\subset W_{x}\subset W_{x}^{*}\subset\lambda(V_{x})$.
Then, if $V_{x_{0}}\cap...\cap V_{x_{n}}\neq\emptyset$, it follows
that $V_{x_{0}}\subset W_{x_{j}}^{*}$ for each $j$ so that $V_{x_{0}}\subset\lambda(V_{x_{j}})$
for each $j$. Therefore, $f(\lambda(V_{x_{0}}),...,\lambda(V_{x_{n}}))|V_{x_{0}}=0$,
so $\lambda^{*}f=0$ in $C^{*}(\mathscr{V},\Gamma)$. Thus, $\check{{H}}^{n}(X,\Gamma)=0$
for all $n$.\end{proof}

\begin{corollary}\label{localisomorphismcorollary}

If $X$ is a paracompact Hausdorff space with basis $\mathcal{B}$
and ${\tau:\Gamma_{1}\to\Gamma_{2}}$ is a local isomorphism of presheaves
on $\mathcal{B}$, then the induced map ${\tau_{*}:\check{{H}}^{*}(X,\Gamma_{1})\to\check{{H}}^{*}(X,\Gamma_{2})}$
is an isomorphism.

\end{corollary}

This gives rise to one important property of the map $\alpha:\Gamma\to(\Gamma^{+})_{\mathcal{B}}$:

\begin{corollary}\label{completioncechcorollary}

If $X$ is a paracompact Hausdorff space with basis $\mathcal{B}$
and $\Gamma$ is a presheaf on $\mathcal{B}$, then the natural homomorphism
${\alpha:\Gamma\to(\Gamma^{+})_{\mathcal{B}}}$ induces a cohomological
isomorphism ${\alpha_{*}:\check{{H}}^{*}(X,\Gamma)\to\check{{H}}^{*}(X,(\Gamma^{+})_{\mathcal{B}})}$.

\end{corollary}

\begin{proof}

By Corollary~\ref{localisomorphismcorollary}, it suffices to show
that $\alpha$ is a local isomorphism. Let $\gamma\in\left[\ker(\alpha)\right](V)$;
then $\gamma\in\Gamma(V)$ and there exists an open covering $\mathscr{U}\subset\mathcal{B}$
of $V$ such that $\gamma|U=0$ for all $U\in\mathscr{U}$. Thus $\ker(\alpha)$
is locally zero.

On the other hand, let $\gamma'\in\left[\text{coker}(\alpha)\right](V)$.
Then there exists an open covering $\mathscr{U}\subset\mathcal{B}$
of $V$ and a compatible $\mathscr{U}$ family $\{\gamma_{U}\}_{U\in\mathscr{U}}$
representing $\gamma'$. For each $U\in\mathscr{U}$, $\gamma'|U$
is represented by $\gamma_{U}\in\alpha(\Gamma(U))$, so $\gamma'|U=0$,
and $\text{coker}(\alpha)$ is also locally zero.\end{proof}

Now, let $\Gamma$ be a presheaf on $X$, $\mathcal{B}$ a basis of
$X$, and $\Gamma_{\mathcal{B}}$ the restriction of $\Gamma$ to
$\mathcal{B}$. Since every open cover $\mathscr{U}$of $X$ has a
refinement $\mathscr{U}'\subset\mathcal{B}$, open covers contained
in $\mathcal{B}$ are cofinal in the collection of all open covers
of $X$. Hence:
\[
\check{{H}}^{*}(X,\Gamma)=\varinjlim\check{{H}}^{*}(\mathscr{U},\Gamma)=\underset{\mathscr{U}\subset\mathcal{B}}{\varinjlim}\check{{H}}^{*}(\mathscr{U},\Gamma)=\underset{\mathscr{U}\subset\mathcal{B}}{\varinjlim}\check{{H}}^{*}(\mathscr{U},\Gamma_{\mathcal{B}})=\check{{H}}^{*}(X,\Gamma_{\mathcal{B}})
\]

Thus, if $\Gamma$ is a presheaf on $X$, its \v{C}ech cohomology
is identical to that of any presheaf on any basis $\mathcal{B}$ to
which $\Gamma$ restricts identically. In this sense, \v{C}ech cohomology
does not really depend on the choice of a basis $\mathcal{B}$, and
unless otherwise specified will be assumed to use the basis consisting
of \emph{all} open subsets of $X$.

If $X$ is a paracompact Hausdorff space and $\Gamma$ is a presheaf
on a basis $\mathcal{B}$, then by Corollary~\ref{completioncechcorollary}
the map $\alpha:\Gamma\to(\Gamma^{+})_{\mathcal{B}}$ induces a natural
isomorphism on \v{C}ech cohomology, so that
\[
\check{{H}}^{*}(X,\Gamma)\cong\check{{H}}^{*}(X,(\Gamma^{+})_{\mathcal{B}})\cong\check{{H}}^{*}(X,\Gamma^{+})
\]

\noindent with the second isomorphism following from our observation
above. Furthermore, on applying this isomorphism twice, we find that:
\[
\check{{H}}^{*}(X,\Gamma)\cong\check{{H}}^{*}(X,\Gamma^{+})\cong\check{{H}}^{*}(X,\Gamma^{++})=\check{{H}}^{*}(X,\hat{\Gamma})
\]

\section{The Continuity Property of \v{C}ech Cohomology}

The usual definition of the continuity property for a cohomology theory
applies to situations where the coefficients of the theory are in
an abelian group or module. We will show that \v{C}ech cohomology
in fact satisfies an even \emph{stronger} version of that property
which can only be defined for a cohomology theory over sheaves of
modules. A few more definitions will be needed first.

Let $\{X_{i},\varphi_{ij}\}_{i\in I}$ be an inverse system of compact
Hausdorff topological spaces, with ${X:=\varprojlim X_{i}}$. We define
a \emph{system of sheaves} $\{\Gamma_{i}\}_{i\in I}$ on this inverse
system to consist of a sheaf of abelian groups or modules $\Gamma_{i}$
on each $X_{i}$ along with injective presheaf maps ${f_{ij}:\Gamma_{i}\to\varphi_{ij}\Gamma_{j}}$
whenever $i\preceq j$ such that:

\textbf{a.} $f_{ii}=\text{id}$ for all $i\in I$

\textbf{b.} $f_{jk}f_{ij}=f_{ik}$ whenever $i\preceq j\preceq k$

Note that whenever $i\preceq j\preceq k$ and $U\subset X_{i}$, we
have
\[
\varphi_{ij}\Gamma_{j}(U)=\Gamma_{j}(\varphi_{ij}^{-1}(U))\overset{f_{jk}}{\to}\varphi_{jk}\Gamma_{k}(\varphi_{ij}^{-1}(U))=\Gamma_{k}(\varphi_{ik}^{-1}(U))=\varphi_{ik}\Gamma_{k}(U)
\]
This defines a morphism $\varphi_{ij}\Gamma_{j}(U)\to\varphi_{ik}\Gamma_{k}(U)$
which will also be denoted $f_{jk}$; then for each $i$, the conditions
on the presheaf maps above make the collection ${\{\varphi_{ij}\Gamma_{j}:i\preceq j\}}$
a direct system.

Let $\mathcal{B}=\{\varphi_{i}^{-1}(U):i\in I\text{ and }U\subset X_{i}\text{ open}\}$
be the standard basis on $X$ as an inverse limit; observe that this
is closed under finite intersections of two or more elements. Then
a \emph{limiting partial presheaf} for the system of sheaves above
is a presheaf $\Gamma$ on $\mathcal{B}$ together with a map of presheaves
$f_{i}:\Gamma_{i}\to\varphi_{i}\Gamma$ for each $i\in I$ such that
$f_{j}f_{ij}=f_{i}$ whenever $i\preceq j$, $\varphi_{i}\Gamma=\underset{i\preceq j}{\varinjlim}\varphi_{ij}\Gamma_{j}$
for each $i$, and $f_{j}:\varphi_{ij}(\Gamma_{j})\to\varphi_{i}(\Gamma)$
is the canonical map into this limit. Such a presheaf exists and is
unique up to isomorphisms of presheaves on $\mathcal{B}$. Similar
to the above, whenever $i\preceq j$, $f_{j}$ gives a map $\varphi_{ij}\Gamma_{j}\to\varphi_{i}\Gamma$
of presheaves on $X_{i}$, which will also be denoted $f_{j}$.

Fix $i\in I$ and let $\mathscr{U}$ be an open cover of $X_{i}$.
Whenever ${i\preceq j}$, ${\varphi_{ij}^{-1}(\mathscr{U}):=\{\varphi_{ij}^{-1}(U):U\in\mathscr{U}\}}$
is an open cover of $X_{j}$, so we may define a map ${\varphi_{ij}^{\#}:\check{{C}}^{*}(\mathscr{U},\varphi_{ij}\Gamma_{j})\to\check{{C}}^{*}(\varphi_{ij}^{-1}(\mathscr{U}),\Gamma_{j})}$
by 
\[
{(\varphi_{ij}^{\#}\theta)(\varphi_{ij}^{-1}(U_{0}),...,\varphi_{ij}^{-1}(U_{n})):=\theta(U_{0},...,U_{n})}
\]
This is well-defined since ${\varphi_{ij}^{-1}(U_{0}\cap...\cap U_{n})}={\varphi_{ij}^{-1}(U_{0})\cap...\cap\varphi_{ij}^{-1}(U_{n})}$,
so ${\varphi_{ij}\Gamma_{j}(U_{0}\cap...\cap U_{n})}={\Gamma_{j}(\varphi_{ij}^{-1}(U_{0})\cap...\cap\varphi_{ij}^{-1}(U_{n}))}$.
Indeed, $\varphi_{ij}^{\#}$ is clearly an isomorphism of chain complexes.
Then, if ${i\preceq j\preceq k}$ the map ${f_{jk}:\varphi_{ij}\Gamma_{j}\to\varphi_{ik}\Gamma_{k}}$
induces a cochain map ${f_{jk}:\check{{C}}^{*}(\mathscr{U},\varphi_{ij}\Gamma_{j})\to\check{{C}}^{*}(\mathscr{U},\varphi_{ik}\Gamma_{k})}$.
This determines a map ${\psi_{jk}:\check{{C}}^{*}(\varphi_{ij}^{-1}(\mathscr{U}),\Gamma_{j})\to\check{{C}}^{*}(\varphi_{ik}^{-1}(\mathscr{U}),\Gamma_{k})}$
so that the following diagram commutes:\medskip{}

\noindent \begin{center}
\begin{tikzcd} {\check{{C}}^{*}(\mathscr{U},\varphi_{ij}\Gamma_{j})} \arrow[r, "f_{jk}"] \arrow[d, "\varphi_{ij}^{\#}"'] & {\check{{C}}^{*}(\mathscr{U},\varphi_{ik}\Gamma_{k})} \arrow[d, "\varphi_{ik}^{\#}"] \\ {\check{{C}}^{*}(\varphi_{ij}^{-1}(\mathscr{U}),\Gamma_{j})} \arrow[r, "\psi_{jk}"']                      & {\check{{C}}^{*}(\varphi_{ik}^{-1}(\mathscr{U}),\Gamma_{k})}                         \end{tikzcd}
\par\end{center}

\noindent \medskip{}

These cochain maps form isomorphic direct systems $\{\check{{C}}^{*}(\mathscr{U},\varphi_{ij}\Gamma_{j}),f_{jk}\}_{i\preceq j}$
and $\{\check{{C}}^{*}(\varphi_{ij}^{-1}(\mathscr{U}),\Gamma_{j}),\psi_{jk}\}_{i\preceq j}$.
For similar definitions of $\varphi_{i}^{\#}$ and $\psi_{j}$, we
obtain cochain maps forming another commutative diagram for $\Gamma$
and $X$ whenever $i\preceq j$:\medskip{}

\noindent \begin{center}
\begin{tikzcd} {\check{{C}}^{*}(\mathscr{U},\varphi_{ij}\Gamma_{j})} \arrow[r, "f_{j}"] \arrow[d, "\varphi_{ij}^{\#}"'] & {\check{{C}}^{*}(\mathscr{U},\varphi_{i}\Gamma)} \arrow[d, "\varphi_{i}^{\#}"] \\ {\check{{C}}^{*}(\varphi_{ij}^{-1}(\mathscr{U}),\Gamma_{j})} \arrow[r, "\psi_{j}"']                      & {\check{{C}}^{*}(\varphi_{i}^{-1}(\mathscr{U}),\Gamma)}                        \end{tikzcd}
\par\end{center}

\noindent \medskip{}

Here $\varphi_{i}^{\#}$ is again an isomorphism, as before. Moreover,
for ${i\preceq j}$ the maps $f_{j}$ are compatible with the first
direct system above, so they determine a cochain map ${f:\underset{i\preceq j}{\varinjlim}\check{{C}}^{n}(\mathscr{U},\varphi_{ij}\Gamma_{j})\to\check{C}^{n}(\mathscr{U},\varphi_{i}(\Gamma))}$.
Similarly, the maps $\psi_{j}$ are compatible with the second direct
system and give a cochain map ${\psi:\underset{i\preceq j}{\varinjlim}\check{{C}}^{n}(\varphi_{ij}^{-1}(\mathscr{U}),\Gamma_{j})\to\check{C}^{n}(\varphi_{i}^{-1}(\mathscr{U}),\Gamma)}$
such that $f$ is an isomorphism if and only if $\psi$ is.

\begin{lemma}\label{Cechcontinuitylemma1}

For each $i\in I$ and finite open cover $\mathscr{U}$ of $X_{i}$,
the maps $f$ and $\psi$ are isomorphisms of chain complexes.

\end{lemma}

\begin{proof}It suffices to show that $f$ is an isomorphism. Suppose
that ${f([\theta])=0}$ where $[\theta]$ is represented by some ${\theta\in\check{C}^{n}(\mathscr{U},\varphi_{ij}\Gamma_{j})}$.
Then ${f_{j}(\theta)=0}$, so for every $(n+1)$\=/tuple ${U:=(U_{0},...,U_{n})}$
of elements of $\mathscr{U}$, ${(f_{j}\theta)(U)=0}$ in ${f_{j}(\varphi_{ij}\Gamma_{j}(U_{0}\cap...\cap U_{n}))\subset\varphi_{i}\Gamma(U_{0}\cap...\cap U_{n})}$.
But ${\varphi_{i}\Gamma=\underset{i\preceq j}{\varinjlim}\varphi_{ij}\Gamma_{j}}$,
so for each such ${U\in\mathscr{U}^{n+1}}$ there exists some ${k_{U}\succeq j}$
with ${(f_{jk_{U}}\theta)(U)=0}$ in ${\varphi_{ik_{U}}\Gamma_{k_{U}}(U_{0}\cap...\cap U_{n})}$.
Since $\mathscr{U}$ is finite, so is $\mathscr{U}^{n+1}$, so there
exists ${k\in I}$ with ${k\succeq k_{U}}$ for \emph{every} ${U\in\mathscr{U}^{n+1}}$.
Then ${(f_{jk}\theta)(U)}={(f_{k_{U}k}f_{jk_{U}}\theta)(U)}=0$ for
\emph{all }${U\in\mathscr{U}^{n+1}}$, so that ${f_{jk}(\theta)=0}$
in $\check{C}(\mathscr{U},\varphi_{ik}\Gamma_{k})$. Thus ${[\theta]=0}$
in $\underset{i\preceq j}{\varinjlim}\check{{C}}^{n}(\mathscr{U},\varphi_{ij}\Gamma_{j})$,
and ${\ker f=0}$.

Next, let ${\zeta\in\check{C}^{n}(\mathscr{U},\varphi_{i}\Gamma)}$.
For each ${U\in\mathscr{U}^{n+1}}$, ${\zeta(U)\in\varphi_{i}\Gamma(U_{0}\cap...\cap U_{n})}$
and ${\varphi_{i}\Gamma=\underset{i\preceq j}{\varinjlim}\varphi_{ij}\Gamma_{j}}$.
Thus there exists ${j_{U}\succeq i}$ and ${\gamma_{U}\in\varphi_{ij_{U}}\Gamma_{j_{U}}(U_{0}\cap...\cap U_{n})}$
with ${f_{j_{U}}(\gamma_{U})=\zeta(U)}$. Again, since $\mathscr{U}^{n+1}$
is finite, we may choose ${k\in I}$ with ${k\succeq j_{U}}$ for
all ${U\in\mathscr{U}^{n+1}}$. Then let ${\theta\in\check{C}(\mathscr{U},\varphi_{ik}\Gamma_{k})}$
be defined by ${\theta(U):=f_{j_{U}k}(\gamma_{U})\in\varphi_{ik}\Gamma_{k}(U_{0}\cap...\cap U_{n})}$,
which makes $(f_{k}\theta)(U)=f_{k}f_{j_{U}k}(\gamma_{U})=f_{j_{U}}(\gamma_{U})=\zeta(U)$
for every ${U\in\mathscr{U}^{n+1}}$. By definition, then, ${f_{k}(\theta)=\zeta}$,
so $\zeta$ is in the image of $f$. Hence $f$ is surjective, and
therefore an isomorphism.\end{proof}

The cochain maps $\psi_{jk}$ induce maps on cohomology \\
${\psi_{jk}:\check{H}^{n}(\varphi_{ij}^{-1}(\mathscr{U}),\Gamma_{j})\to\check{H}(\varphi_{ik}^{-1}(\mathscr{U}),\Gamma_{k})}$
which form a direct system of abelian groups $\{\check{H}^{*}(\varphi_{ij}^{-1}(\mathscr{U}),\Gamma_{j}),\psi_{jk}\}_{i\preceq j}$.
Similarly, the cochain maps $\psi_{j}$ induce a compatible family
of morphisms ${\psi_{j}:\check{H}^{n}(\varphi_{ij}^{-1}(\mathscr{U}),\Gamma_{j})\to\check{H}^{n}(\varphi_{i}^{-1}(\mathscr{U}),\Gamma)}$
which determine a map ${\psi:\underset{i\preceq j}{\varinjlim}\check{H}^{n}(\varphi_{ij}^{-1}(\mathscr{U}),\Gamma_{j})\to\check{H}^{n}(\varphi_{i}^{-1}(\mathscr{U}),\Gamma)}$.
From Lemma~\ref{Cechcontinuitylemma1} we immediately obtain:

\begin{corollary}\label{Cechcontinuitycorollary1}

The map $\psi$ on cohomology is an isomorphism for all ${n\in\mathbb{N}}$.

\end{corollary}

For each ${j\succeq i}$ and open cover $\mathscr{U}$ of $X_{i}$,
define ${h_{ij}:\check{H}^{n}(X_{i},\Gamma_{i})\to\check{H}^{n}(X_{j},\Gamma_{j})}$
to be the direct limit of the maps ${\psi_{ij}=\varphi_{ij}^{*}f_{ij}:\check{H}(\mathscr{U},\Gamma_{i})\to\check{H}^{n}(\varphi_{ij}^{-1}(\mathscr{U}),\Gamma_{j})}$
over the collection of all open covers $\mathscr{U}$ of $X_{i}$.
These maps form a direct system of abelian groups $\{\check{H}^{n}(X_{i},\Gamma_{i}),h_{ij}\}_{i\in I}$.
Furthermore, by taking the direct limit of the morphisms ${\psi_{i}=\varphi_{i}^{*}f_{i}:\check{H}^{n}(\mathscr{U},\Gamma_{i})\to\check{H}^{n}(\varphi_{i}^{-1}(\mathscr{U}),\Gamma)}$
over the same collection we obtain a compatible family of morphisms
${h_{i}:\check{H}^{n}(X_{i},\Gamma_{i})\to\check{H}(X,\Gamma)}$ for
each ${i\in I}$. This family of maps determines a homomorphism ${h:\varinjlim\check{H}^{n}(X_{i},\Gamma_{i})\to\check{H}^{n}(X,\Gamma)}$.

\begin{theorem}\label{Cechcontinuitytheorem1}

If $\{X_{i},\varphi_{ij}\}_{i\in I}$ is a surjective inverse system
of compact Hausdorff spaces, then $h$ is an isomorphism for all $n$.

\end{theorem}

\begin{proof}

To see that $h$ is injective, let ${\theta\in\check{H}^{n}(X_{i},\Gamma_{i})}$
satisfy ${h_{i}(\theta)=0}$. Since $X_{i}$ is compact, the collection
of finite open covers is cofinal for the system determining $\check{H}^{n}(X_{i},\Gamma_{i})$,
so $\theta$ is represented by a cohomology class ${[\theta_{i}]\in\check{H}^{n}(\mathscr{U},\Gamma_{i})}$
for some finite open cover $\mathscr{U}$ of $X_{i}$. Therefore,
${\psi_{i}([\theta_{i}])\in\check{H}^{n}(\varphi_{i}^{-1}(\mathscr{U}),\Gamma)}$
is taken to $0$ in the limit $\check{H}^{n}(X,\Gamma)$, so there
exists some ${j\in I}$ and finite open cover $\mathscr{V}$ of $X_{j}$
such that $\varphi_{j}^{-1}(\mathscr{V})$ is a refinement of $\varphi_{i}^{-1}(\mathscr{U})$
and the restriction map ${\check{H}^{n}(\varphi_{i}^{-1}(\mathscr{U}),\Gamma)\to\check{H}^{n}(\varphi_{j}^{-1}(\mathscr{V}),\Gamma)}$
takes $\psi_{i}([\theta_{i}])$ to $0$. 

Now, choose some ${k\succeq i,j}$. Since $\varphi_{ik}^{-1}(\mathscr{U})$
and $\varphi_{jk}^{-1}(\mathscr{V})$ are open covers of $X_{k}$
and $X_{k}$ is compact, there exists a finite open cover $\mathscr{W}$
of $X_{k}$ which is a common refinement of both $\varphi_{ik}^{-1}(\mathscr{U})$
and $\varphi_{jk}^{-1}(\mathscr{V})$. Thus $\varphi_{k}^{-1}(\mathscr{W})$
is a refinement of ${\varphi_{k}^{-1}(\varphi_{jk}^{-1}(\mathscr{V}))=\varphi_{k}^{-1}(\mathscr{V})}$,
hence a refinement of $\varphi_{i}^{-1}(\mathscr{U})$ also, giving
the following commutative diagram, where $r$ is the restriction map:\medskip{}

\noindent \begin{center}
\begin{tikzcd}[column sep=-1.8em, row sep=3em]                                                                                          & {\check{H}^{n}(\mathscr{U},\Gamma_{i})} \arrow[ld, "\psi_{ik}"'] \arrow[rd, "\psi_i"] &                                                                    \\ {\check{H}^{n}(\varphi_{ik}^{-1}(\mathscr{U}),\Gamma_{k})} \arrow[d, "r"'] \arrow[rr, "\psi_k"] &                                                                                       & {\check{H}^{n}(\varphi_i^{-1}(\mathscr{U}),\Gamma)} \arrow[d, "r"] \\ {\check{H}^{n}(\mathscr{W},\Gamma_{k})} \arrow[rr, "\psi_k"]                                    &                                                                                       & {\check{H}^{n}(\varphi_k^{-1}(\mathscr{W}),\Gamma)}                \end{tikzcd}
\par\end{center}

\noindent \medskip{}

In this diagram, the restriction ${\check{H}^{n}(\varphi_{i}^{-1}(\mathscr{U}),\Gamma)\to\check{H}(\varphi_{k}^{-1}(\mathscr{W}),\Gamma)}$
maps $\psi_{i}([\theta_{i}])$ to $0$. Let $[\theta_{k}]$ be the
image under the restriction map \\
${\check{H}^{n}(\varphi_{ik}^{-1}(\mathscr{U}),\Gamma_{k})\to\check{H}(\mathscr{W},\Gamma_{k})}$
of $\psi_{ik}([\theta_{i}])$. Then ${\psi_{k}([\theta_{k}])=0}$
in $\check{H}^{n}(\varphi_{k}^{-1}(\mathscr{W}),\Gamma)$. However,
by Corollary~\ref{Cechcontinuitycorollary1}, ${\check{H}^{n}(\varphi_{k}^{-1}(\mathscr{W}),\Gamma)\cong\underset{k\preceq\ell}{\varinjlim}\check{H}^{n}(\varphi_{k\ell}^{-1}(\mathscr{W}),\Gamma_{\ell})}$.
Thus, for some ${\ell\succeq k}$, ${\psi_{k\ell}([\theta_{k}])=0}$
in $\check{H}(\varphi_{\ell}^{-1}(\mathscr{W}),\Gamma_{\ell})$, so
${h_{i\ell}(\theta)=0}$ in $\check{H}^{n}(X_{\ell},\Gamma_{\ell})$.
Therefore, ${\ker(h)=0}$ and $h$ is injective.

To see that $h$ is surjective, note that each element $\zeta$ of
$\check{H}^{n}(X,\Gamma)$ is represented by a cohomology class $[\zeta_{i}]\in\check{H}^{n}(\varphi_{i}^{-1}(\mathscr{U}),\Gamma)$
for some $i\in I$ and some finite open cover $\mathscr{U}$ of $X_{i}$.
By Corollary~\ref{Cechcontinuitycorollary1}, there exists some $j\succeq i$
and $[\theta_{i}]\in\check{H}(\varphi_{ij}^{-1}(\mathscr{U}),\Gamma_{j})$
such that $\psi_{j}([\theta_{i}])=[\zeta_{i}]$. Then the equivalence
class $\theta\in\check{H}^{n}(X_{j},\Gamma_{j})$ represented by $[\theta_{i}]$
has $h_{j}(\theta)=[\psi_{j}([\theta_{i}])]=\zeta$. Therefore, $h$
is surjective, and so an isomorphism.\end{proof}

Define the \emph{limiting sheaf} of the system of sheaves to be the
sheaf $\hat{\Gamma}$ on $X$ which is the sheafification of $\Gamma$.
By composing $h$ with the isomorphism $\check{{H}}^{*}(X,\Gamma)\cong\check{{H}}^{*}(X,\hat{\Gamma})$
described in the discussion following Corollary~\ref{completioncechcorollary},
we obtain the following:

\noindent \begin{corollary}\label{Cechcontinuitycorollary2}

For a surjective inverse system of spaces with a system of sheaves
$\Gamma_{i}$, if $\hat{\Gamma}$ is the limiting sheaf, then 
\[
\varinjlim\check{H}^{n}(X_{i},\Gamma_{i})\cong\check{H}^{n}(X,\hat{\Gamma})
\]

\noindent \end{corollary}

\section{Application to Profinite Group Cohomology}

Let ${G=\varprojlim G_{i}}$ be a profinite group, specifically the
inverse limit of a surjective system of finite discrete groups $G_{i}$
indexed by a directed set $I$, and let $\varphi_{i}:G\to G_{i}$
be the projection maps. Let $A$ be any discrete $G$\=/module. For
any $i\in I$, consider the open normal subgroup $\ker(\varphi_{i})$
of $G$ and $A_{i}:=A^{\ker(\varphi_{i})}$, the subgroup of $A$
consisting of elements which are invariant under the action of $\ker(\varphi_{i})$.
The action of $G$ on $A$ determines an action of $G_{i}$ on $A_{i}$;
furthermore the collection $\{A_{i}\}_{i\in I}$ with inclusion maps
naturally forms a direct system whose limit is $A$. Then, as observed
in \cite{MR2599132}, the cohomology groups $H^{k}(G,A)$ may be computed
using the cohomology groups $H^{k}(G_{i},A_{i})$; specifically there
is a family of natural direct systems $\{H^{k}(G_{i},A_{i})\}_{i\in I}$
such that $H^{k}(G,A)\cong\varinjlim H^{k}(G_{i},A_{i})$ for all
$k\geq0$.

On the other hand, since each $G_{i}$ is a discrete finite group,
it has the classifying space $BG_{i}$ with the properties that $\pi_{1}(BG_{i})\cong G_{i}$
and the cohomology groups $H_{G_{i}}^{k}(BG_{i},A_{i})$ with local
coefficients given by the action of $G_{i}$ on $A_{i}$ are naturally
isomorphic to $H^{k}(G_{i},A_{i})$ for every ${k\geq0}$; see, for
example, \cite{MR1867354}. One way to construct the classifying space
$BG_{i}$ is through simplicial sets, as described in, for example,
\cite{MR232393,MR208595}. This construction takes the geometric realization
of the nerve $NG_{i}$ of $G_{i}$, and the universal cover $EG_{i}$
of $BG_{i}$ results from taking the geometric realization of another,
closely related simplicial set, which we will call $\mathcal{E}G_{i}$.

Let $NG_{i}^{\leq n}$ be the $n$\=/truncation of $NG_{i}$, and
note that since $G_{i}$ is finite, $NG_{i}[k]$ is also finite for
every $k$, so that each $NG_{i}^{\leq n}$ is a simplicial finite
$n$\=/truncated set. The inverse system $\{G_{i},\varphi_{ij}\}$
of which $G$ is a limit induces surjective maps ${\varphi_{ij}:NG_{j}\to NG_{i}}$
which clearly restrict to maps ${\varphi_{ij}:NG_{j}^{\leq n}\to NG_{i}^{\leq n}}$,
and in both cases satisfy the necessary conditions to form a surjective
inverse system, so in particular we have the inverse system $\{NG_{i}^{\leq n},\varphi_{ij}\}$
of simplicial finite \emph{n}\=/truncated sets. For each pair $i$,
$n$, the geometric realization $|NG_{i}^{\leq n}|$ is the $n$-skeleton
$BG_{i,n}$ of the classifying space $BG_{i}=|NG_{i}|$; in particular
it is a finite CW-complex. Thus, this space is compact, Hausdorff,
path-connected, and locally contractible. Hence if $n>1$, it follows
that $\pi_{1}(|NG_{i}^{\leq n}|)\cong G_{i}$ and $H_{G_{i}}^{k}(|NG_{i}^{\leq n}|,A_{i})\cong H_{G_{i}}^{k}(|NG_{i}|,A_{i})\cong H^{k}(G_{i},A_{i})$
for all ${k<n}$. On the other hand, the action of $G_{i}$ on $A_{i}$
determines a locally constant sheaf $\mathcal{A}_{i}^{n}$ on $|NG_{i}^{\leq n}|$.
Since $|NG_{i}^{\leq n}|$ is paracompact, we obtain a natural isomorphism
$\check{{H}}^{k}(|NG_{i}^{\leq n}|,\mathcal{A}_{i}^{n})\cong H_{G_{i}}^{k}(|NG_{i}^{\leq n}|,A_{i})$
through sheaf cohomology; this is shown in, for example, Chapter 3
of \cite{MR1481706}.

If $X$ is a locally contractible, path-connected, paracompact and
Hausdorff space with the universal cover $p:\tilde{X}\to X$, let
$\mathscr{U}=\{U_{x}:x\in X\}$ be an open cover indexed by $X$ where
each $U_{x}$ is chosen such that $p^{-1}(U_{x})$ is a union of disjoint
open sets in $\tilde{X}$ which are each mapped homeomorphically onto
$U_{x}$ by $p$. Call such a cover a \emph{$p$\=/cover of $X$}.
Now, if $\mathscr{V}$ is any open cover indexed by $X$, we may choose
any $p$\=/cover $\mathscr{U}$ and take the common refinement $\mathscr{U}\cap\mathscr{V}:=\{U_{x}\cap V_{x}:x\in X\}$,
which then has $\mathscr{U},\mathscr{V}\leq\mathscr{U}\cap\mathscr{V}$.
Furthermore, each $U_{x}\cap V_{x}\subset U_{x}$ has $p^{-1}(U_{x}\cap V_{x})$
as a union of disjoint open sets of $\tilde{X}$ which map homeomorphically
onto $U_{x}\cap V_{x}$, so this is a $p$\=/cover. Hence the collection
of $p$\=/covers is cofinal in the collection of all open covers
indexed by $X$, so that $p$\=/covers alone can be used to determine
the \v{C}ech cohomology of $X$.

\begin{lemma}\label{profiniteclassifyingspace2lemma}

For a fixed $n$, the sheaves $\{\mathcal{A}_{i}^{n}\}_{i\in I}$
form a system of sheaves on the inverse system of spaces $\{|NG_{i}^{\leq n}|\}_{i\in I}$.

\end{lemma}

\begin{proof}

First note that, since each $|NG_{i}^{\leq n}|$ has the universal
cover \\
${EG_{i}^{\leq n}=|\mathcal{E}G_{i}^{\leq n}|}$, the maps $\varphi_{ij}$
further determine maps $\varphi_{ij}:EG_{j}^{\leq n}\to EG_{i}^{\leq n}$
such that ${\varphi_{ij}p_{j}=p_{i}\varphi_{ij}}$ where $p_{i}:EG_{i}^{\leq n}\to|NG_{i}^{\leq n}|$
are the universal covering maps. Let $\overline{A}_{i}^{n}$ be the
constant presheaf associated with $A_{i}$ on the space $EG_{i}^{\leq n}$
and let $\hat{A}_{i}^{n}$ be its sheafification, the constant sheaf.
Then for any open subset $U$ of $EG_{i}^{\leq n}$, the map $f_{ij}:A_{i}\to A_{j}$
determines an injective map $f_{ij}:\overline{A}_{i}^{n}(U)\to\varphi_{ij}\overline{A}_{j}^{n}(U)$
as follows: If $U$ is empty, this is just the $0$ map; otherwise
the source is $A_{i}$ and the target is $A_{j}$, so the map is precisely
$f_{ij}:A_{i}\to A_{j}$. Since the restriction maps to nonempty subsets
are just the identity, these maps clearly commute with restriction
maps, as needed. This in turn induces an injective presheaf map $\hat{A}_{i}^{n}\to\widehat{\varphi_{ij}A_{j}^{n}}$
between the sheafifications of the sheaves (on $EG_{i}^{\leq n}$);
compose this with the natural map $\widehat{\varphi_{ij}A_{j}^{n}}\to\varphi_{ij}\hat{A}_{j}^{n}$
to obtain an injective presheaf map $\hat{f}_{ij}:\hat{A}_{i}^{n}\to\varphi_{ij}\hat{A}_{j}^{n}$.

Next, by definition, if $U\subset|NG_{i}^{\leq n}|$, then
\[
\mathcal{A}_{i}^{n}(U)=\{\gamma\in\hat{A}_{i}^{n}(p_{i}^{-1}(U)):\text{for every }x_{i}\in p_{i}^{-1}(U)\text{ and }g_{i}\in G_{i},\gamma(g_{i}\cdot x_{i})=g_{i}\cdot\gamma(x_{i})\}
\]
and, similarly,
\[
\mathcal{A}_{j}^{n}(\varphi_{ij}^{-1}(U))=\{\gamma\in\hat{A}_{j}^{n}(p_{j}^{-1}\varphi_{ij}^{-1}(U)):\text{for every }x_{j}\in p_{j}^{-1}\varphi_{ij}^{-1}(U)
\]
\[
\text{ and }g_{j}\in G_{j},\gamma(g_{j}\cdot x_{j})=g_{j}\cdot\gamma(x_{j})\}
\]
Since $\varphi_{ij}p_{j}=p_{i}\varphi_{ij}$, $p_{j}^{-1}\varphi_{ij}^{-1}(U)=\varphi_{ij}^{-1}p_{i}^{-1}(U)$,
so $\hat{f}_{ij}$ descends to an injective map $f_{ij}:\mathcal{A}_{i}^{n}(U)\to\varphi_{ij}\hat{A}_{j}^{n}(p_{i}^{-1}(U))=\hat{A}_{j}^{n}(\varphi_{ij}^{-1}p_{i}^{-1}(U))=\hat{A}_{j}^{n}(p_{j}^{-1}\varphi_{ij}^{-1}(U))$;
we claim that the image of this map is contained in $\varphi_{ij}\mathcal{A}_{j}^{n}(U)$.
Let $x_{i}$ be any element of $|NG_{i}^{\leq n}|$ and $U_{i}$ an
open neighborhood of $x_{i}$ such that $p_{i}^{-1}(U_{i})$ is a
union of disjoint open sets in $EG_{i}^{\leq n}$ which are each mapped
homeomorphically onto $U_{i}$ by $p_{i}$. Then, similarly, every
element $x_{j}\in\varphi_{ij}^{-1}(x_{i})$ (which is nonempty since
$\varphi_{ij}$ is surjective) has $\varphi_{ij}^{-1}(U_{i})$ as
an open neighborhood with $p_{j}^{-1}\varphi_{ij}^{-1}(U_{i})$ a
union of disjoint open sets in $EG_{j}^{\leq n}$ each of which is
mapped homeomorphically onto $\varphi_{ij}^{-1}(U_{i})$ by $p_{j}$.
Hence, $\mathcal{A}_{i}^{n}$ restricts to a constant sheaf on $U_{i}$,
and $f_{ij}(\mathcal{A}_{i}^{n}|U_{i})$ is a constant sheaf on $\varphi_{ij}^{-1}(U_{i})$,
but also $f_{ij}(\mathcal{A}_{i}^{n}(U_{i}))\leq\hat{A}_{j}^{n}(p_{j}^{-1}\varphi_{ij}^{-1}(U_{i}))=\mathcal{A}_{j}^{n}(\varphi_{ij}^{-1}(U_{i}))=\varphi_{ij}\mathcal{A}_{j}^{n}(U_{i})$
since $\mathcal{A}_{j}^{n}$ restricts to a constant sheaf on $\varphi_{ij}^{-1}(U_{i})$.

Next, let $U$ be \emph{any} open subset of $|NG_{i}^{\leq n}|$.
By the comments preceding this proof, for any open cover $\mathscr{U}$
of $U$ there exists a $p_{i}$\=/cover refinement $\mathscr{V}$
of $\mathscr{U}$, and by the above discussion each $V\in\mathscr{V}$
satisfies ${f_{ij}(\mathcal{A}_{i}^{n}(V))\leq\mathcal{A}_{j}^{n}(\varphi_{ij}^{-1}(V))}$.
Since $\hat{A}_{j}^{n}$ is a sheaf, each element $f_{ij}(\gamma)\in f_{ij}(\mathcal{A}_{i}^{n}(U))$
may be identified with a compatible $\mathscr{V}$\=/family $\{\eta_{V}\}$
where each ${\eta_{V}\in f_{ij}(\mathcal{A}_{i}^{n}(V))\leq\mathcal{A}_{j}^{n}(\varphi_{ij}^{-1}(V))}$.
Since $\mathcal{A}_{j}^{n}$ is \emph{also} a sheaf, this family uniquely
determines an element ${\eta\in\mathcal{A}_{j}^{n}(\varphi_{ij}^{-1}(V))}$
with $f_{ij}(\gamma)=\eta$. Thus ${f_{ij}(\mathcal{A}_{i}^{n}(U))\leq\mathcal{A}_{j}^{n}(\varphi_{ij}^{-1}(U))}$
also, as needed.

Finally, the conditions that $f_{ii}$ is the identity and ${f_{jk}f_{ij}=f_{ik}}$
whenever ${i\preceq j\preceq k}$ follow directly from the same conditions
holding on the system $\{A_{i}\}_{i\in I}$ of modules.\end{proof}

Let $\mathcal{A}^{n}$ be the limiting sheaf of this system on $BG^{n}$.
Then we have:

\begin{theorem}\label{profiniteclassifyingspace2theorem}

For all $k<n$, 
\[
\check{{H}}^{k}(BG^{n},\mathcal{A}^{n})\cong\varinjlim\check{{H}}^{k}(|NG_{i}^{\leq n}|,\mathcal{A}_{i}^{n})\cong\varinjlim H^{k}(G_{i},A_{i})\cong H^{k}(G,A)
\]

\end{theorem}

\begin{proof}

Since by Lemma~\ref{profiniteclassifyingspace2lemma} the sheaves
$\{\mathcal{A}_{i}^{n}\}_{i\in I}$ form a system of sheaves on the
inverse system of spaces $\{|NG_{i}^{\leq n}|\}_{i\in I}$ and $\mathcal{A}^{n}$
is that system's limiting sheaf, Corollary~\ref{Cechcontinuitycorollary2}
gives an isomorphism $\check{{H}}^{k}(BG^{n},\mathcal{A}^{n})\cong\varinjlim\check{{H}}^{k}(|NG_{i}^{\leq n}|,\mathcal{A}_{i}^{n})$.
Next, we have the natural isomorphisms $\check{{H}}^{k}(|NG_{i}^{\leq n}|,\mathcal{A}_{i}^{n})\cong H_{G_{i}}^{k}(|NG_{i}^{\leq n}|,A_{i})=H_{G_{i}}^{k}(BG_{i},A_{i})\cong H^{k}(G_{i},A_{i})$;
hence these isomorphisms commute with both direct systems, so they
give an isomorphism $\varinjlim\check{{H}}^{k}(|NG_{i}^{\leq n}|,\mathcal{A}_{i}^{n})\cong\varinjlim H^{k}(G_{i},A_{i})$
between their limits. Finally, we apply the isomorphism $\varinjlim H^{k}(G_{i},A_{i})\cong H^{k}(G,A)$
from \cite{MR2599132} as discussed above.\end{proof}

\bibliographystyle{amsalphawithkeys}
\nocite{*}
\bibliography{CCPbibtex}

\end{document}